\documentclass[11pt,leqno]{article}

\usepackage{amssymb,amsmath,amsthm}

\newcommand{\bb}[1]{\mathbb{#1}}
\newcommand{\cl}[1]{\mathcal{#1}}
\newcommand{\vp}{\varepsilon}

\newcommand{\sst}{\scriptstyle}

\newcommand{\ms}{\medskip}

\theoremstyle{plain}
\newtheorem{thm}{Theorem}[section]

\newtheorem{lem}[thm]{Lemma}
\newtheorem{cor}[thm]{Corollary}

\theoremstyle{remark}

\newtheorem{rem}[thm]{Remark}

\theoremstyle{definition}

\newtheorem{stp}{Step}

\numberwithin{equation}{section}

\def\R{\bb R}

\def\C{\bb C}

\def\P{\bb P}

\def\N{\bb N}

\def\P{\bb P}

\def\ie{{\it i.e.\ }}

\def\phi{\varphi}

\begin{document}

\title{The non-commutative Khintchine inequalities for $0<p<1$}

\author{by\\
 Gilles Pisier and \'Eric Ricard\footnote{Partially supported by  
    ANR-2011-BS01-008-01.}\\
Texas A\&M University 
and  
 Universit\'e  de Caen Basse-Normandie}
 \date{ }

\maketitle

\begin{abstract} 
We give a proof of the   Khintchine inequalities
  in  non-commutative $L_p$-spaces  for all $0< p<1$. 
These new inequalities are  valid for the Rademacher functions or Gaussian random variables, but also for more general   sequences, e.g. for the   analogues of such random variables in free  probability. We also   prove
  a factorization for operators
  from a Hilbert space to a non commutative $L_p$-space, which is new 
  for $0<p<1$. We end by showing that Mazur maps are H\"older on semifinite von Neumann algebras. \end{abstract}

2000 MSC 46L51, 46L07, 47L25, 47L20
\vfill\eject
The Khintchine inequalities for   non-commutative $L_p$-spaces were
 first proved by Lust-Piquard in \cite{LP1}.
 They play an   important r\^ole in the recent developments
in non-commutative Functional Analysis, and in particular in Operator Space Theory, see \cite{P2}. Just like their commutative counterpart for ordinary $L_p$-spaces, they are    a crucial tool to understand
the behavior of unconditionally convergent series of random
variables, or random vectors,  
in non-commutative $L_p$ (\cite{PX}). The commutative version is 
closely   related to Grothendieck's Theorem (see \cite{Ma,Ma2}).
Moreover, in the non-commutative case, 
 Random Matrix Theory and
Free Probability provide further ground for applications 
of the non-commutative Khintchine inequalities.
For instance, they imply the  remarkable fact that the Rademacher functions
(i.e. i.i.d. $\pm 1$-valued independent random  variables) satisfy the same inequalities
as the freely independent ones in non-commutative $L_p$ for $p<\infty$.
See \cite{DR} for a recent direct simple proof of the free version
of these inequalities, which   extend to $p=\infty$. 

 In the most classical setting, the  non-commutative Khintchine inequalities  deal with Rademacher series
 of the form
 $$S=\sum\nolimits_k r_k(t) x_k$$
 where $(r_k)$ are the Rademacher functions on the Lebesgue interval
 (or any independent symmetric sequence of random choices of signs)
 where the coefficients $x_k$ are in the Schatten $q$-class $S_q$
 or in a non-commutative $L_q$-space associated to a semifinite trace $\tau$.
 Let us denote simply by $\|.\|_q$ the norm (or quasi-norm) in the latter  Banach
 (or quasi-Banach) space, that we will denote by $X$. By Kahane's well known results, this series converges
 almost surely in norm iff it converges in $L_q(dt;X)$
 (and in fact in $L_p(dt;X)$ for any $0<p<\infty$, but for obvious reasons we prefer to work in the present context with $p=q$). Thus to characterize
 the almost surely   norm-convergent series such as $S$, it suffices
 to produce a two sided equivalent of $\|S\|_{L_q(dt;X)}$, and this is 
 precisely what   the non-commutative Khintchine inequalities provide:\\
 For any $0<q<\infty$ there are positive constants
 $\alpha_q,\beta_q$
such that for any finite set $(x_1,\ldots, x_n)$ in $X=S_q$ (or $X=L_q(\tau)$)
we have
$$ \frac 1 {\beta_q} |||(x_k) |||_q\le  \left(\int \big\|S(t)\big\|^q_q dt\right)^{1/q}\le \alpha_q |||(x_k) |||_q$$
where 
$|||(x_k) |||_q$ is defined as follows:\\
 {  If } $2\le q<\infty$
 $$|||(x_k)|||_q = \max\left\{\left\|\left(\sum x^*_kx_k\right)^{\frac 12} \right\|_q , 
 \left\|\left(\sum x_k x^*_k\right)^{\frac 12} \right\|_q\right\}$$
 {and if }  $0< q\le 2$:
\begin{equation}\label{eq0.2}
|||x|||_q \overset{\sst\text{def}}{=} \inf_{x_k=a_k+b_k} \left\{\left\|\left(\sum a^*_ka_k\right)^{\frac 12} \right\|_q + \left\|\left(\sum b_kb^*_k\right)^{\frac 12}\right\|_q\right\}.
\end{equation}
Note that $\beta_q=1$ if $q\ge 2$, while $\alpha_q=1$ if $q\le 2$ and the corresponding one sided bounds are easy. The difficulty is to verify
the other side.

 The case $1<q<\infty$ is due to Lust-Piquard   \cite{LP1}.
 The case $q=1$ was proved   (in two ways) in \cite{LPP}, together with a new proof of $1<q<\infty$. This also implied the fact (independently observed by Junge) that $\alpha_q=O(\sqrt{q})$ when $q\to \infty$, which yielded an interesting
 subGaussian estimate. Later on,  Buchholz proved in \cite{Buc} a sharp
 version valid when  $q>2$ is any even integer, the best $\alpha_q$ happens to be the same as in the commutative (or scalar) case.

The case  $q<2$  of the Khintchine inequalities has a more delicate formulation, 
 but this case can be handled easily when $1<q<2$ using a suitable duality argument. The case $q=1$ is closely related to the ``little non-commutative
 Grothendieck inequality" in the sense of \cite{P7}
 (first proved in \cite{P1}): actually, one of the proofs given for that case in \cite{LPP} shows that it is essentially ``equivalent" to it.
 More recently, Haagerup and Musat (\cite{HM1}) gave a new proof
 that yields the best constant (equal to 2) for $q=1$ for
 the complex analogue (namely Steinhaus random variables)
 of the Rademacher functions.

 In \cite{P5}
 the first named author proved by an extrapolation argument
 that the validity of this kind of inequalities for some $1<q<2$ implies
 their validity for all $1\le p<q$, but the  case $q<1$ remained open.
However, very recently
   the second named author noticed that
the method proposed in \cite{P5} actually works
in this case too. The latter method reduced the problem
to a certain form of H\"older type inequality which could not be verified
because the  required  ingredients   
 (duality and triangular projection) became seemingly unavailable for  
$0<q<1$. In \cite{P5} a certain very weak form of the required 
H\"older type estimate
was identified as sufficient to complete the case $q<1$. It is this form
that the second named author was able to establish
by an a priori ultraproduct argument (see Remark \ref{ultra}). 
Although his argument failed to produce
explicitly a quantitative estimate, it showed that some estimate does exist.
The goal of this paper is to produce   an explicit estimate,
and a reasonably self-contained proof of the case $q<1$.
 In fact, it turns out that a certain
  version of H\"older's inequality (perhaps of independent interest) does hold, thus we can produce an explicit
  estimate, similar to the case $q\ge 1$ but
with unexpected exponents. This inequality, namely
\eqref{riceq2.0b} below, may prove useful in the theory of means developed
in \cite{HK,HK2}.
 
In the rest of the paper we will consider only the case $0<q\le 2$. 
In that case, 
 our      inequalities reduce to this : There is $\beta_q$
 such that for any finite sequence
  $(x_k)$    in an arbitrary non-commutative $L_q$-space, we have
  \begin{equation}\label{eq0.1}
|||x|||_q \le \beta_q\left(\int\left\|\sum r_k(t) x_k\right\|^q_q dt\right)^{\frac 1q}
\end{equation}
where $|||x|||_q$ is as in \eqref{eq0.2}.
 
In this paper,  as in \cite{P5}, we will show that the validity of \eqref{eq0.1} for some 
fixed $q$ with $1<q<2$ implies its validity (with another constant) for all 
value of $q$ in $(0,q)$ (and in particular for all $q$ in $(0,1)$).  
For that deduction the only assumption needed on $(r_k)$ is its orthonormality in $L_2([0,1])$. Thus our approach yields \eqref{eq0.1} also for more general sequences than the Rademacher functions. For instance, we may apply it to
free Haar unitaries in the sense of \cite{VDN} or to  the ``$Z(2)$-sequences'' considered in \cite{Har}.  

In \S \ref{sec4} we prove an extension to the case $0<q<1$ of
the ``little Grothendieck inequality", i.e. a (Maurey type)  factorization for bounded linear maps from a Hilbert space to a non-commutative $L_p$-space.

In \S \ref{sec5} we extend to the case $0<p<1$ some of the results of \cite{R}
giving H\"older exponents for the Mazur map
 $M_{p,q}:L_p(\tau)\to L_q(\tau)$  given by $$M_{p,q}(f)=f|f|^{\frac
   {p-q}q},$$
   relative to a semifinite von Neumann algebra $(M,\tau)$.

For convenience, we recall   an elementary fact: if 
  $X$ is an $L_p$-space (commutative or not)
  and if  $0<p\le 1$, the
quasi-norm $\|\ \|$ of $X$ satisfies the ``$p$-triangle inequality":
 \begin{equation}\label{tri2}\forall x,y\in X\quad \|x+y\|^p\le \|x\|^p+\|y\|^p.\end{equation}
Actually, it will be convenient to invoke 
 a  consequence of the triangle inequality, valid,
 this time, for all $0<p\le \infty$:
 \begin{equation}\label{tri}
 \forall x,y\in X\quad \|x+y\|\le \chi_p (\|x\|+\|y\|)\le 2\chi_p \max\{\|x\|, \|y\|\},
 \end{equation}
 where
 $$\chi_p=\max\{2^{\frac{1}{p}-1},1\}.$$

\section{The case $\pmb{1\le q<2}$ from \cite{P5}}\label{sec1}

In this section, we   review
(and partly reproduce)  the previous attempt from \cite{P5} to explain the
contribution of the present paper.\\
 Here, $L_2([0,1])$   can be replaced by any non-commutative $L_2$-space $L_2(\varphi)$ associated to a semifinite generalized (i.e.\ ``non-commutative'') measure space,
and $(r_k)$ is then replaced by an
orthonormal sequence $(\xi_k)$ in $L_2(\varphi)$. Then
     the right-hand side of \eqref{eq0.1} is replaced by $$\|\sum \xi_k\otimes x_k\|_{L_q(\varphi\otimes\tau)}.$$
More precisely, by a (semifinite) generalized measure space $(N,\varphi)$ we mean a von~Neumann algebra $N$ equipped with a faithful, normal, semifinite trace $\varphi$.  Without loss of generality,
we may always reduce consideration to the $\sigma$-finite case.
Throughout this paper, we will use freely the basics of non-commutative integration as described in \cite{Ne} or \cite[Chap. IX]{Tak2}.

Let us fix another generalized measure space $(M,\tau)$. The inequality we are interested in now takes the following form:
\[
(K_q)\quad \left\{
\begin{array}{l}
\exists \beta_q \text{ such that for any finite sequence}\\
x = (x_k) \text{ in } L_q(\tau) \text{ we have}\\
|||x|||_q \le \beta_q \left\|\sum \xi_k\otimes x_k\right\|_{L_q(\varphi\otimes\tau)}\\
\text{where } |||\cdot|||_q \text{ is defined as in \eqref{eq0.2}.}
\end{array}\right.
\]

In the Rademacher case, i.e. when $(\xi_k)=(r_k) $,   
we refer to these as
the non-commutative Khintchine inequalities.

We can now state   the main result of \cite{P5} for the case $q\ge 1$.

\begin{thm}\label{thm1.1}\cite{P5}
Let $1<q<2$. Recall that $(\xi_k)$ is assumed orthonormal in $L_2(\varphi)$. \\ Then $(K_q)\Rightarrow (K_p)$ for all $1\le p<q$.
\end{thm}

Here is a sketch of the argument in \cite{P5}. We denote
\[
 S = \sum \xi_k\otimes x_k.
\]
Let ${\cl D}$ be the collection of all ``densities,'' i.e.\ all $f$ in $L_1(\tau)_+$ with $\tau(f) = 1$. Fix $p$ with $0<p\le q$. Then we denote for $x=(x_k)$ 
\begin{equation}\label{ric2}
C_q(x) = \inf\left\{\left\|\sum \xi_k\otimes y_k\right\|_q\right\}
\end{equation}
where $\|\cdot\|_q$ is the norm in $L_q(\varphi\otimes\tau)$ and the infimum runs over all sequences $y = (y_k)$ in $L_q(\tau)$ for which there is $f$ in ${\cl D}$ such that
\[
 x_k = (f^{\frac1p-\frac1q}y_k + y_kf^{\frac1p-\frac1q})/2.
\]
Note that $C_p(x) = \|S\|_p$.

The proof of Theorem \ref{thm1.1} is based on a variant of ``Maurey's extrapolation principle'' (see \cite{Ma}) This combines three steps:\ (here $C',C'',C''',\ldots$ are constants independent of $x=(x_k)$ and we wish to emphasize that here $p$ remains fixed while the index $q$ in $C_q(x)$ is such that $p<q\le2$).\ms

\begin{stp}\label{stp1}
Assuming $(K_q)$ we have
\[
 |||x|||_p \le C'C_q(x).
\]
\end{stp}

\begin{stp}\label{stp2}
\[
C_2(x) \le C''|||x|||_p.
\]
Actually   the converse inequality also holds (up to a constant), see \cite{P5}.\ms 
\end{stp}

\begin{stp}\label{stp3}  
\[
 C_q(x) \le C'''C_p(x)^{1-\theta} C_2(x)^{\theta},
\]
where $\theta$ is defined by $\frac{1}{q}=\frac{1-\theta}{p} +\frac{\theta}{2}$. 
The three steps put all together yield
\begin{align*}
 |||x|||_p &\le C'C''' C_p(x)^{1-\theta} (C''|||x|||_p)^{\theta}\\
\intertext{and hence}
|||x|||_p &\le C^{''''}C_p(x) = C^{''''}\|S\|_p.
\end{align*}
\end{stp}
 
 Only the proof of Step 3 required $p\ge 1$ in \cite{P5}.
 Note that actually it suffices  that,
 for some $0<\theta'<1$, we have
\begin{equation}\label{sup5}
 C_q(x) \le C'''C_p(x)^{1-\theta'} C_2(x)^{\theta'},
\end{equation}
and we will show that this essentially holds in \S \ref{sec2}, but with a 
rather surprising exponent given  by
$$1-\theta'= (1-\theta)\frac{R }{2},$$
where $R$ is any number such that $0<R<p$.
As will be explained below in Remark \ref{r-substi}, the bound in
\eqref{sup5} can be deduced from the following  variant of H\"older's inequality,
that  will be proved in \S 2: There is a
 constant $c$ such that  
\begin{equation}\label{sup5bis}
\forall x\in L_2(\tau)\ \forall f\in \cl D\quad\|f^{\alpha(1-\theta)} x + xf^{\alpha(1-\theta)}\|_q\le c  \|f^\alpha x + xf^\alpha\|_p^{1-\theta'} \|x\|_2^{\theta'}.\end{equation}
When $p\ge 1$ this holds (see \cite{P5}) with $\theta'=\theta$.
In the commutative case (or if there is only one term), when $\theta'=\theta$ this reduces to H\"older's inequality
for $L_p$-norms (just write $f^{\alpha(1-\theta)} x=(f^{\alpha}x)^{1-\theta} x^{\theta}$ and recall $\frac 1q=\frac {1-\theta}{p}+\frac{\theta}2$), so this holds (with $c=2^\theta$) for $0<p<\infty$.\\
When $p>1$ the (complete) boundedness of the triangular projection on $S_p$ implies
\begin{equation}\label{sup5ter}\max\{\|f^\alpha x\|_p,  \|xf^\alpha\|_p\} \lesssim \|f^\alpha x + xf^\alpha\|_p,\end{equation} from which \eqref{sup5bis} with $\theta'=\theta$ is immediate (see
\cite{JuP} or  \cite[1.9 (iii)]{P5}).
However this fails for $p\le 1$, because,  
 by a well known argument, such an estimate would imply conversely the boundedness of the triangular projection, which fails for $p\le 1$.  \\
When $p<1$ we do not know whether \eqref{sup5}  or \eqref{sup5bis} holds with  $\theta'=\theta$.

   \begin{rem}\label{rem1.3}
In Theorem \ref{thm1.1}, the assumption that $(\xi_k)$ is orthonormal in $L_2(\varphi)$ 
  can be replaced by the following one:\ for any finite sequence $y = (y_k)$ in $L_2(M,\tau)$ we have
\begin{equation}\label{eq1.5}
 \left\|\sum \xi_k\otimes y_k\right\|_{L_2(\varphi\otimes\tau)} \le \left(\sum \|y_k\|^2_2\right)^{\frac 12}.
\end{equation}
 \end{rem}

We will need the following fact. Results of this kind originate in Arazy and Friedman's memoir \cite{AF} and can also be found in
Junge and Parcet's paper \cite{JuP} (see also \cite{HK2} for related inequalities).

\begin{lem}[\cite{RX1}]\label{lem1.6} Let $Q_j$ $(j=1,\ldots,n)$ be mutually orthogonal projections
in $M$  and let $\lambda_j$ $(j=1,\ldots,n)$ be non-negative numbers.
There is a constant $C$ so that for any $1\leq q\leq \infty$ and $\theta\in[0,1]$, for any $x$ in $L_q(\tau)$
   \[
\frac 1 C\|x\|_{L_q(\tau)}\leq  \left\|\sum^{n}_{i,j=1} \frac{\lambda_i^\theta+\lambda_j^\theta}{(\lambda_i+\lambda_j)^\theta}Q_i xQ_j\right\|_{L_q(\tau)} \le C\|x\|_{L_q(\tau)} \]
\end{lem}

\begin{rem}
If $0<p\le 2$, the converse inequality to $(K_p)$ is valid assuming that
$\varphi(1)=1$ and $\xi_k\in L_2(N,\varphi)$ is orthonormal or satisfies \eqref{eq1.5}.
Indeed, for any $t\ge 0$ in $N\otimes M$ since $\frac p2\le 1$ and $\varphi(1)=1$, by the operator concavity of
$t\mapsto t^{\frac p2}$ (see \cite[p. 115-120]{Bh}), 
 we have 
$
 \|t\|_{\frac p2} \le \|{\bb E}^M(t)\|_{\frac p2}
$
and hence, if $S = \sum \xi_k\otimes x_k$, we have
$$
 \|S\|_p = \|S^*S\|^{\frac 12}_{\frac p2}  \le \|{\bb E}^M(S^*S)\|^{\frac 12}_{\frac p2}\\
 \le \left\|\left(\sum x^*_kx_k\right)^{\frac 12}\right\|_p,
$$
and similarly
\[
 \|S\|_p \le \left\|\left(\sum x_kx^*_k\right)^{\frac 12}\right\|_p.
\]
From this we easily deduce\[
 \|S\|_p \le \chi_p |||x|||_p.
\]
 where $\chi_p$ is as in \eqref{tri}. 
\end{rem}

\begin{rem}\label{r-substi} To extend Theorem \ref{thm1.1} to the case $0<p<1$ the difficulty lies in Step \ref{stp3}, or in proving
a certain form of H\"older inequality. As
mentioned in \cite{P5}, actually a much \emph{weaker} estimate allows to conclude:\\
It suffices to show that there is a function $\vp\mapsto\delta(\vp)$ tending to zero with $\vp>0$ such that when $f\in {\cl D}$ we have $(\alpha=\frac1p - \frac12=\frac1r)$ $(1<q<2)$:
\begin{equation}\label{substi}    [\|x\|_2 \le 1, \|f^\alpha x + xf^\alpha\|_p\le \vp]\Rightarrow
 \|f^{\alpha(1-\theta)} x + xf^{\alpha(1-\theta)}\|_q \le \delta(\vp).
\end{equation}
When $1<p<\infty$, we have by \cite{P5}, $\delta(\vp)\leq C_p \vp^{1-\theta}$. This estimate can already be found in \cite[Ch. 3]{AF}  for Schatten classes.
We will show that \eqref{substi}  can be substituted to Step \ref{stp3}. 

Indeed,
setting $w(\vp)=2\delta(\vp) \vp^{-1}$, by homogeneity this implies
$$ \|f^{\alpha(1-\theta)} x + xf^{\alpha(1-\theta)}\|_q \le \delta(\vp)
\max\{ \|x\|_2  , \vp^{-1}\|f^\alpha x + xf^\alpha\|_p\}\le \delta(\vp)\|x\|_2  +\frac{w(\vp)}2\|f^\alpha x + xf^\alpha\|_p. $$

Fix $\vp'>0$. Let $y_k$ be such that $x_k = (f^{\alpha}y_k+y_kf^{\alpha})/2$ with
\[
 \left\|\sum \xi_k\otimes y_k\right\|_2 < C_2(x)(1+\vp').
\]

We will again denote
$S=\sum \xi_k\otimes x_k$, and we set $Y=\sum \xi_k\otimes y_k$.

Let us assume that $(M,\tau)$ is $M_n$ equipped with usual trace (the argument works  assuming merely that   $f$
has finite spectrum). We will use the orthonormal basis for which $f$ is diagonal with coefficients denoted by $(f_i)$. We have then
\[
 (y_k)_{ij} = 2(f^{\alpha}_i + f^{\alpha}_j)^{-1}(x_k)_{ij}.
\]
Then the above inequality, with $Y$ in place of $x$, yields
 \begin{equation}\label{ric1}
\|f^{\alpha(1-\theta)} Y + Y f^{\alpha(1-\theta)}\|_q \le \delta(\vp)\|Y\|_2  +w(\vp)\|S\|_p. \end{equation}
We will now compare  the elements $T$ and $Z$ defined by
$$T=f^{\alpha(1-\theta)} Y + Y f^{\alpha(1-\theta)}=[ f_i^{\alpha(1-\theta)} Y_{ij} + Y_{ij} f_j^{\alpha(1-\theta)}]=2[\frac{f_i^{\alpha(1-\theta)}+f_j^{\alpha(1-\theta)} }
{ f_i^{\alpha}+f_j^{\alpha}} x_{ij} ],$$

$$Z=[\frac{2}{ f_i^{\alpha \theta}+f_j^{\alpha\theta}} S_{ij} ]=
[\frac{f_i^\alpha+f_j^\alpha}{ (f_i^{\alpha \theta}+f_j^{\alpha\theta})(f_i^{\alpha (1-\theta)}+f_j^{\alpha(1-\theta)})} T_{ij} ] .$$
Using Lemma \ref{lem1.6} twice, we find
 \begin{equation}\label{sup1}\|Z\|_q\le C^2 \|T\|_q=C^2 \|f^{\alpha(1-\theta)} Y + Y f^{\alpha(1-\theta)}\|_q.\end{equation}
Now,
since $S=(f^{\alpha\theta} Z +Z f^{\alpha\theta})/2$ and $\alpha\theta=\frac1p-\frac1q$,
by
 \eqref{ric2}  we have $C_q(x)  \le  \|Z\|_q$  and \eqref{sup1} implies
 $$C_q(x)  \le C^2 (\delta(\vp)\|Y\|_2  +w(\vp)\|S\|_p),$$
and since the inf of $\|Y\|_2$ over all factorizations of the form $x_k = (f^{\frac 1r}y_k+y_kf^{\frac 1r})/2$ (or equivalently $S= (f^{\frac 1r}Y+Yf^{\frac 1r})/2$)
is equal to $C_2(x)$
we find
\begin{equation}\label{ric13} C_q(x)\le C^2 (\delta(\vp)C_2(x)  +w(\vp)\|S\|_p)\end{equation}
by Step  1  and 2, for some $c$
$$ |||x|||_p \le c (\delta(\vp)|||x|||_p   +w(\vp)\|S\|_p)$$
 and hence choosing $\vp$ small enough we again conclude
 $$ |||x|||_p \le c'  \|S\|_p .$$
 
 The preceding arguments, up to \eqref{ric13}, work just as well if we merely assume that $f$ has finite spectrum.
 We now use this to complete the proof in the general semifinite case.\\
 Let $(y_k)$ and $f\in \cl D$ be such that
 $ x_k = (f^{\alpha}y_k+y_kf^{\alpha})/2$ 
 and 
 $(\sum \|y_k\|_2)^{\frac 12}<2C_2(x).$ Recalling Step 2, we have
$$(\sum \|y_k\|_2)^{\frac 12} \le 2C'' |||x|||_p.$$
 Fix $\vp'>0$. Let $g$ be a density
  with finite spectrum such that $\|g-f\|_1<\vp'$. We can find such a $g$ 
  by approximating the spectral decomposition of $f$ so that
  the spectral decomposition of
  $g$ commutes with that of $f$.  
  Then, for any $\beta>0$, we have clearly
  a bound $\|g^\beta-f^\beta\|_{\frac 1\beta}\le o(\vp')$,  and hence, if we wish, we can find a density $g$ such that we actually  have 
 \begin{equation}\label{ric14}\|f^{\alpha}-g^{\alpha}\|_{\frac 1\alpha}<\vp'.\end{equation}
 
 Let $x'_k=(g^{\alpha} y_k+y_k g^{\alpha})/2$ and $S'=\sum \xi_k \otimes x_k'$. Note
 \begin{equation}\label{ric16}C_2(x')\le (\sum \|y_k\|_2)^{\frac 12} \le 2C'' |||x|||_p.\end{equation}
 By the proof of
 \eqref{ric13}  applied with $g$ in place of $f$
 we find
 $$C_q(x')\le C^2 (\delta(\vp)C_2(x')  +w(\vp)\|S'\|_p).$$
 By Step 1 we have
 $$|||x'|||_p\le C' C_q(x')\le C' C^2 (\delta(\vp)C_2(x')  +w(\vp)\|S'\|_p)$$
 and hence by \eqref{ric16}
 \begin{equation}\label{ric15} |||x'|||_p\le C' C^2 (2C''\delta(\vp) |||x|||_p  +w(\vp)\|S'\|_p).\end{equation}
 But clearly by H\"older  and \eqref{ric14}, we have an estimate $\|x_k-x_k'\|_p\le o(\vp')$
 and hence we have both $|||x-x'|||_p\le o(\vp')$ and $\|S-S'\|_p \le o(\vp')$. Thus, letting $\vp'\to 0$,  we deduce from \eqref{ric15}
 that
$$ |||x|||_p  \le C' C^2 (2C''\delta(\vp) |||x|||_p  +w(\vp)\|S\|_p).$$
 and we conclude as before that
 $$ |||x|||_p \le c'  \|S\|_p .$$
\end{rem}

\begin{rem}\label{ultra}

We give a sketch of a proof of \eqref{substi} using an ultraproduct argument.
Clearly by a $2\times2$ trick, we may assume that $x=x^*$.
Assuming  \eqref{substi} does not hold gives some $\varepsilon>0$, a sequence of elements 
$x_n\in L_2(M,\tau)$, $f_n\in \mathcal D$ with $\|x_n\|_2=1$, $x_n=x_n^*$
with $\|f_n^\alpha x_n + x_nf_n^\alpha\|_p\le \frac 1 n$ but $\|f_n^{\alpha(1-\theta)} x_n + x_nf_n^{\alpha(1-\theta)}\|_q \geq \varepsilon$.

 We use the theory of ultrapowers from \cite{Ray}. In the latter,
 Theorem 3.6 explains that given a free ultrafilter $\mathfrak U$ on $\N$,
 there is a general (type III) von Neumann algebra $\mathcal A$ so
 that there are natural identifications $\prod_{\mathfrak U}
 L_p(M,\tau)=L_p(\mathcal A)$ for $p>0$. Of course, taking powers and
 products commutes with the ultrapower construction (see Theorems 3.6
 and 5.1 in \cite{Ray}).

 Consider 
$x=(x_n)\in L_2(\mathcal A)$, $f=(f_n)\in L_2(\mathcal A)$. We have 
$x=x^*$ with $\|x\|_2=1$, $f\geq 0$ with $\|f\|_1=1$ and $\|f^\alpha x + xf^\alpha\|_p=0$ but $\|f^{\alpha(1-\theta)} x + xf^{\alpha(1-\theta)}\|_q \geq \varepsilon$.

From the definition of $L_p$-spaces associated to a type III von
Neumann algebra (see \cite{PX, terp}), $f$ and $x$ can be seen as
$\tau$-measurable operators associated to the core $\tilde {\mathcal
  A}$ of $\mathcal A$ which is semifinite with trace $\tau$. Recall
that the $\tau$-measurable operators $L_0(\tilde {\mathcal A},\tau)$
form a topological $*$-algebra, we have $f^\alpha x=-xf^\alpha$. Hence
$f^{2\alpha} x=-f^\alpha xf^\alpha=xf^{2\alpha}$, as $x=x^*$ so that
$f^{2\alpha}$ and $x$ strongly commute (see Lemma 2.3 in \cite{Kos3}).
Thus $x$ commutes with any spectral projection of $f^{\alpha}$. But
spectral projections of $f$ and $f^t$ coincide for any $t>0$, we get
that $x$ and $f^t$ commute.  We have $\|2xf^\alpha\|_p=0$, but for any
spectral projection $p=1_{(a,b)}(f)$ with $0<a<b<\infty$, there is some 
$v\in \tilde {\mathcal A}$ with $f^\alpha v=pf^{\alpha(1-\theta)}$, so $x f^{\alpha(1-\theta)}p=0$. Letting $a\to 0$ and $b\to \infty$, $pf^{\alpha(1-\theta)}$ converges to 
$f^{\alpha(1-\theta)}$ in $L_0(\tilde {\mathcal A},\tau)$, hence 
$xf^{\alpha(1-\theta)}=0$. This contradicts $\|2xf^{\alpha(1-\theta)}\|_q\geq \varepsilon$.

 The trace $\tau$ on $M$ does not play any r\^ole in the above argument. Thus \eqref{substi} holds for any type III von Neumann algebra $M$ with $f\in L_1(M)^+$ and 
$\|f\|_1=1$.  

\end{rem}

\section{The new case $\pmb{0<p<1}$ }\label{sec2}

The proofs in this section are valid
for $0<p <2$ but are really pertinent only for 
$0<p<1$. For simplicity, to avoid distinguishing
the normed case from the $p$-normed one,
 we assume $0<p<1$ throughout. We will compensate for the lack of convexity
 with   subharmonicity. Indeed, it is well known that
 on an $L_p$-space, commutative or not,
 the  norm, as well as the function $x\mapsto \|x\|^p_p$,  is subharmonic.
 We will use moreover certain inequalities which express its
  ``uniform subharmonicity", in analogy with the uniform
  convexity of $L_p$ when $p>1$.

  Let $0<p<s\le \infty$.
  In this section,  we set $$\alpha=\frac 1r =\frac 1 p-\frac 1 s.$$
  The previous section corresponds to the particular value  $s=2$.
  
  Let $x$ be in  $L_s(\tau)$, and let $f\in L_1^+$ with $\|f\|_1=1$.
  \def\a{\alpha}
  Note that $\|f^\a\|_r=1$.
  Let $0<\theta<1$. Let $q$ be determined by
  $$\frac{1}{q}=\frac{1-\theta}{p} +\frac{\theta}{s}.$$
  
  Our main result is a new form of non-commutative H\"older inequality:
  
  \begin{thm}\label{nho} Let  $0<p<q<s\le \infty$. Let $\alpha, \theta$ be as above. Then for any $0<R<p$ there is a constant $C$ such that
  for any  $x\in L_s(\tau)$ and  $f\in L_1(\tau)^+$ with $\|f\|_1=1$,
  and for any unitaries $V,\,W\in M$ commuting with $f$
  we have
  \begin{equation}\label{riceq2.0}\big\|x Wf^{\a(1-\theta)} +V f^{\a(1-\theta)}  x\big\|_q\le C\big\| x Wf^{\a } +V f^{\a }  x   \big\|_p^{\frac{R}{2}(1-\theta)}
  \|x\|_s ^{1-\frac{R}{2}(1-\theta)}.\end{equation}
  In particular for any choice of sign $\pm 1$ we have
 \begin{equation}\label{riceq2.0b}\big\|x f^{\a(1-\theta)} \pm f^{\a(1-\theta)}  x\big\|_q\le C\big\| x f^{\a } \pm f^{\a }  x   \big\|_p^{\frac{R}{2}(1-\theta)}
  \|x\|_s ^{1-\frac{R}{2}(1-\theta)}.\end{equation}
  \end{thm}
  
  Since this implies that \eqref{substi} holds with $\delta(\vp)=O(\vp^{\frac{R}{2}(1-\theta)}
)$ (with $s=2$), by Remark
  \ref{r-substi}
  we have
  \begin{cor} The implication
  $K_q\Rightarrow K_p$ remains valid for any $0<p<1$.
  In particular the non-commutative Khintchine inequality
  \eqref{eq0.1} holds for any $0<q<1$.
  \end{cor}
 
  \begin{cor} 
  Let $0<p(1),p(2)<\infty$ and let $u:\ L_{p(1)}(\tau)\to L_{p(2)}(\tau)$ be any bounded linear operator.
  Then there is a constant $C(p(1),p(2))$ such that for any
finite sequence $(x_j)$ in   $L_{p(1)}(\tau)$ we have
$$|||(u(x_j))|||_{p(2)} \le C(p(1),p(2))\ \|u\|\, |||(x_j)|||_{p(1)} .$$
  \end{cor}
    \begin{proof} Since this is clear when $|||(x_j)|||_{p(1)}$ and $|||(u(x_j))|||_{p(2)} $
    are replaced by the corresponding Rademacher averages,  this corollary follows
    from the Khintchine inequalities, now extended to the whole range
    $0<p<\infty$. Note that it is well known that the Kahane inequalities
    remain valid for quasi-normed spaces.
   \end{proof}
   
   We denote by $U=\{z\in \C\mid 0<\Re(z)<1\}$ the classical  vertical strip
  of unit width of the complex plane.
  
  \begin{rem} It seems worthwhile to start by a rough outline 
   of the proof of Theorem \ref{nho}.
   The natural way to prove \eqref{riceq2.0} (and we will use this in the end)
    is to introduce the analytic function
    $G_0( z)= xWf^{\alpha(1-z)} + V f^{\alpha(1-z)}x$ defined for $z\in U$
    and to use some form of the 3-line lemma, as in \eqref{riceq2.9}  below, 
    to estimate $\|G_0(\theta)\|_q$. In order to do so, we need to 
    have 
    bounds on the two  boundary vertical lines. The bound 
    for $z=1+it$ of the form  $\|G_0(1+it)\|_s\le c\|x\|_s$ 
    is straightforward. The problem is 
    the missing bound  $$\|G_0(it)\|_p\le c\|G_0(0)\|_p \quad ?$$
    which seems highly unrealistic. However, it turns out 
    that using the complex uniform convexity 
    (and the simple algebraic form of $G_0$) as in  \eqref{riceq2.3},
    we will be able to majorize, for any fixed $\gamma>1$, the function
    $$    G( z)=G_0(\gamma z+1-\gamma)= x W f^{\gamma \alpha(1-z)} + V f^{\gamma\alpha(1-z)}x,$$
    so that we have for    $0<\omega<1$ and $0<R<p$ such that $\frac1q= \frac{1-\omega}{R} + \frac{\omega}{s}  $
    $$\Big(\int \big\|G(it)\big\|_R^R Q^0_\omega(dt)\Big)^{\frac 1R} \lesssim\big\|x\big\|_s^{1-\frac{R}{2}} \big\|G_0(0)\big\|^{\frac{R}{2}}_p.$$
    Again we have $\|G(1+it)\|_s\le c\|x\|_s$. Denote $1-\omega= \frac {1-\theta}\gamma$ so that $G( {\omega} )=G_0(\theta)$.
    Thus, applying the 3-line type argument (see \eqref{riceq2.9}  below)   to  the function $G$,   we obtain a bound of the form
  $$\big\|G_0(\theta)\big\|_q =\big\|G( {\omega} )\big\|_q\lesssim \Big(  \big\|x\big\|_s^{1-\frac{R}{2}} \big\|G_0(0)\big\|^{\frac{R}{2}}_p\Big)^{1-\omega} \big\|x\big\|_s^\omega$$
which yields \eqref{riceq2.0}.

   \end{rem}
   We will use the well known fact that complex interpolation
  remains valid for the  $L_p(\tau)$ spaces in the range $0<p<1$.
  We will need just one direction.
  (Note however that the argument given for this fact
  for Schatten classes at the end of \cite{CPP}
  is erroneous.) For simplicity we restrict the discussion here to the finite case.

  \begin{lem}\label{cpp} Assume given $(M,\tau)$ as before with $\tau$ finite.
  Let $0<p_0<p_1\le \infty$.
  Let $G$ be a bounded analytic function on $U$ with
  values in $L_{p_1}(\tau)$,  
  admitting a.e. non-tangential boundary values. 
  Let us set for $j=0,1$
  $$g(j+it)=\big\|G(j+it)\big\|_{L_{p_j}(\tau)}.$$
  For any $0<\theta <1$, let $p_\theta$ be such that $\frac{1}{p_\theta}=\frac{1-\theta}{p_0}+\frac{\theta}{p_1}$.
 Let $\P_U^\theta$ denote the probability  measure
which is the  harmonic measure
associated to $\theta$ with respect to  $U$.
  We have then
   \begin{equation}\label{riceq2.8}\log \|G(\theta)\|_{L_{p_\theta}(\tau)   }\le \int_{\partial U} 
   \log g(\xi)\  \P_U^\theta(d\xi).\end{equation}
Let $Q^0_\theta$ (resp. $Q^1_\theta$)
 be  the probability on $\{\xi\in \C\mid \Re(\xi)=0\}$ (resp.    $\{\xi\in \C\mid \Re(\xi)=1\}$) such that $\P_U^\theta=(1-\theta) Q^0_\theta  + \theta Q^1_\theta$, we have then
  \begin{equation}\label{riceq2.9}  \|G(\theta)\|_{L_{p_\theta}(\tau)   }\le \left(\int g(it)^{p_0} \ Q^0_\theta(dt)\right)^{\frac {1-\theta} {p_0}} \left( \int g(1+it)^{p_1} \ Q^1_\theta(dt)\right)^{\frac \theta {p_1}} .\end{equation}
  \end{lem}
  \begin{proof}[Sketch] We freely use conformal equivalence
  with the disc to justify the technical points.  
    \\ The inequality \eqref{riceq2.8} is well known when $p_0\ge 1$.
  We will show that if it holds for a pair $p_0,p_1$
  then it also holds for the pair $\frac {p_0}2,\frac {p_1}2$. This clearly suffices to cover the whole range $0<p_0<p_1\le \infty$. \\
      Fix $\vp>0$.
  Let $$\forall \xi \in \partial U\quad w(\xi)= (G(\xi)^* G(\xi))^{\frac 12} +\vp 1=|  G(\xi)|+\vp 1.$$
Following a well established tradition, we claim that $G$ admits a factorization
as a product  of analytic functions on $U$:
$$G=G_1G_2$$
satisfying $$G_1G_1^* \le |G^*|=(GG^*)^{\frac 12}
\text{ and } G_2^*G_2 \le w \text{ on }  \partial U  .$$
Indeed, by the classical operator valued analogue
of Szeg\"o's theorem (see e.g. Th. 8.1 in \cite{PX} and the Remark (i) after it), there is a bounded analytic
function $F$ with values in $L_{2p_1}(\tau)$
such that
$$\forall \xi \in \partial U\quad F(\xi)^* F(\xi)= w(\xi).$$
Moreover, since $w\ge \vp 1$, 
the function $z \mapsto F(z)^{-1}$ is well defined and bounded on $U$.
Let then 
$G_1=G F^{-1}$ and $G_2=F$.    
Let $G=u|G|$ be the polar decomposition. Then, as is classical, we have
$GG^*=u|G|^2u^*$ and $|G^*|=u|G|u^*$. 
Moreover, 
since $\lambda \mapsto \lambda^{\frac 12} (\lambda+\vp)^{-1} \lambda^{\frac 12}  $ is at most $1$ on $\R_+$, we clearly have $|G|^{\frac 12} w^{-1} |G|^{\frac 12} \le 1$.
Thus
we have on $\partial U$
$$G_1G^*_1= G (F^*F)^{-1} G^*=G w^{-1} G^*=u|G|w^{-1}|G| u^* \le u|G|u^*= |G^*|.$$
This proves our claim.\\
Let $g_1(j+it)=\big\|G_1(j+it)\big\|_{L_{p_j}(\tau)}$
and $g_2(j+it)=\big\|G_2(j+it)\big\|_{L_{p_j}(\tau)}$.
Assume the Lemma holds   for the pair $p_0,p_1$.
Then $G_1$ and $G_2$ satisfy \eqref{riceq2.8}. Therefore
since by H\"older
$$\big\|G(\theta)\big\|_{L_{  {p_\theta}/2}(\tau)  } \le \big\|G_1(\theta)\big\|_{L_{p_\theta}(\tau)  }\big\|G_2(\theta)\big\|_{L_{p_\theta}(\tau)  }
$$
and also
$ \big\|G_1(j+it)\big\|^2_{L_{p_j}(\tau)}=\big\|G(j+it)\big\|_{L_{\frac {p_j}2}(\tau)} $
and $ \big\|G_2(j+it)\big\|^2_{L_{p_j}(\tau)}=\big\|w(j+it)\big\|_{L_{ \frac {p_j}2}(\tau)} $,
if we add the two inequalities \eqref{riceq2.8} written for $G_1$ and $G_2$
we obtain 
$$2\log \big\|G(\theta)\big\|_{L_{ {p_\theta}/2}(\tau)   }\le (1-\theta) \int 
   \log \big\|G(it)\big\|_{L_{  {p_0}/2}(\tau)} Q^0_\theta(dt)
   +\theta \int 
   \log \big\|G(1+it)\big\|_{L_{  {p_1}/2}(\tau)} Q^1_\theta(dt)$$
   $$ +          (1-\theta) \int 
   \log \big\|w(it)\big\|_{L_{  {p_0}/2}(\tau)} Q^0_\theta(dt)
   +\theta \int 
   \log \big\|w(1+it)\big\|_{L_{ {p_1}/2}(\tau)} Q^1_\theta(dt) .$$  
   Letting $\vp\to 0$, we obtain \eqref{riceq2.8},
   and \eqref{riceq2.9} follows, as in the classical case, using the convexity
   of the exponential function. 
     \end{proof}
We will use a certain form of uniform convexity
inequalities for Hardy spaces with values in non-commutative $L_p$-spaces for $0<p\le 1$,
extending  the case $p=1$ which is treated in \cite{HP}. See \cite{DGT} for more
on complex uniform convexity (and in particular Haagerup's inequality
included as Th. 4.3 in \cite{DGT}).  We refer the interested reader to  \cite{TJ}
  for early estimates of the moduli
  of uniform convexity of the Schatten classes $S_p$ for $1<p\le 2$. 
  See also \cite{BCL} and more recently \cite{RX2} for optimal constants in the
  associated martingale inequalities.
  The next result appears as Th. 3.1 in \cite{X2}. 

\begin{thm}[Q. Xu] Let $0<p\le 2$. There is a constant $ \delta_p>0$
such that for any function $F\in H_p(D;L_p(\tau))$
we have 
\begin{equation}\label{riceq2.1}\big\|F(0)\big\|^2_{L_p(\tau)}
+\delta_p \big\|F-F(0)\big\|^2_{H_p(D;L_p(\tau))}\le \big\|F\big\|^2_{H_p(D;L_p(\tau))}.\end{equation}
  \end{thm}

    \begin{rem}
  A similar inequality holds
  for the values $F(\zeta)$ of $F$ at another point $\zeta$ of $D$.
  Indeed, using a M\"obius map as conformal equivalence
  taking $0$ to $\zeta$, we find
  \begin{equation}\label{riceq2.2}\big\|F(\zeta)\big\|^2_{L_p(\tau)}
+\delta_p \big\|F-F(\zeta)\big\|^2_{L_p(\mu^{\zeta};L_p(\tau))}\le \big\|F\big\|^2_{L_p(\mu^{\zeta};L_p(\tau))},\end{equation}
where $\mu^{\zeta}$ denotes the Poisson probability (harmonic) measure
on $\partial D$ associated to $\zeta\in D$.
  \end{rem}
 
  \begin{rem} By conformal equivalence,
  a similar inequality holds
  with the unit strip $$U=\{z\in \C\mid 0<\Re(z)<1\}$$  in place of  the unit disc $D$.
 For any  $0<\theta<1$, recall that  $\P_U^\theta$ denotes the probability  measure
which is the Poisson or harmonic measure
associated to $\theta$ with respect to the strip $U$. Then
\eqref{riceq2.2} becomes
   \begin{equation}\label{riceq2.3}\big\|F(\theta)\big\|^2_{L_p(\tau)}
+\delta_p \big\|F-F(\theta)\big\|^2_{L_p(\P_U^{\theta};L_p(\tau))}\le \big\|F\big\|^2_{L_p(\P_U^{\theta};L_p(\tau))}.\end{equation}
  
  \end{rem}
  
  \begin{proof}[Proof of Theorem \ref{nho}]
  First we reduce the situation where $x=x^*$, indeed let us assume the result hold in this situation and consider the $2\times 2$-matrices
\begin{equation}\label{trick}
\tilde W =\left[\begin{array}{cc} V^* & 0 \\ 0 &
    W\end{array}\right], \tilde V =\left[\begin{array}{cc} V & 0
    \\ 0 & W^*\end{array}\right], \tilde f =\frac 1 2
\left[\begin{array}{cc} f & 0 \\ 0 & f\end{array}\right], \tilde x
=\left[\begin{array}{cc} 0 & x \\ x^* &0 \end{array}\right].
\end{equation} Those
elements satisfy the assumptions in $M_2(M)$ and $\tilde x$ is
selfadjoint. But one has
$$\big\|\tilde x\tilde W \tilde f^{\alpha(1-\theta)}+\tilde V \tilde f^{\alpha(1-\theta)}
\tilde x\big\|_q=2^{\frac 1 q-1} \big\| x W  f^{\alpha(1-\theta)}+ V  f^{\alpha(1-\theta)}
 x\big\|_q,$$$$
\big\|\tilde x\tilde W \tilde f^{\alpha}+\tilde V \tilde f^{\alpha}
\tilde x\big\|_p=2^{\frac 1 p-1} \big\| x W  f^{\alpha}+ V  f^{\alpha}
 x\big\|_p.$$
So that \eqref{riceq2.0} for $x, \,f,\, V,\,W$ follows from that of 
 $\tilde x, \,\tilde f,\, \tilde V,\,\tilde W$.

Next we reduce the proof to finite von Neumann algebras. To see that,
 assume the result is true for finite von Neumann algebras. 
 Let $p_n=1_{(\frac 1 n,\infty)}(f)$. Note that $p_n$ commutes with $f$ and $V$. We have
    $p_n\to p_\infty=1_{(0,\infty)}(f)$ in $M$ for the strong operator topology. This implies that
    for any $t<\infty$ and $y\in L_t(\tau)$ $\|p_nyp_n -p_\infty y p_\infty\|_{t}\to 0$ (and 
$\|p_nyp_n\|_\infty\leq \|y\|_\infty$ if $y\in M$).
As $p_nMp_n$ is finite, we can apply the result to $p_nxp_n$, $fp_n$ and $Vp_n$ and let $n\to \infty$. Note that $f(1-p_\infty)=(1-p_\infty)f=0$. There is a remaining term
of the form $(1-p_\infty) x Wf^{\a(1-\theta)}p_\infty +p_\infty V f^{\a(1-\theta)}  x (1-p_\infty) $ which is easy to handle, it splits in two terms
that can be treated using basic one sided estimates.

Next we reduce the proof to the technically easier case when $f$
  has a finite spectrum. Fix $\vp'>0$. Let $g\in L_1^+$
  with finite spectrum such that $\|g-f\|_1<\vp'$. We can find such a $g$ 
  by approximating the spectral decomposition of $f$ so that
  spectral decomposition of
  $g$ commutes with $f$.  
  Then, for any $\beta>0$, we have clearly
  a bound $\|g^\beta-f^\beta\|_{\frac 1\beta}\le o(\vp')$, from which it follows,
  by H\"older, that for any $y\in L_s(\tau)$ we have
  $\|y(g^\beta-f^\beta)W+V(g^\beta-f^\beta)y\|_{  s +\frac 1\beta}\le o(\vp')$.
  A fortiori, we have
  $$\left|\|yg^\beta W +V g^\beta y\|_{  s +\frac 1\beta} - \|y Wf^\beta+V f^\beta y\|_{  s +\frac 1\beta} \right |\le o(\vp).$$
  From this last inequality it becomes clear
  that we may reduce the proof of \eqref{riceq2.0}
  to the case when $f$
  has a finite spectrum and $x=x^*$, so we assume this in the rest of the proof.

  Fix $1<\gamma$. We will apply  \eqref{riceq2.3} to the analytic function
  $F:\ U\to  L_R(\tau)$
  defined by $$F(\zeta)= f^{\gamma\a \zeta}  x  f^{\gamma\a(1-\zeta)  },$$
where $\frac 1R = \frac 1 s + \a \gamma$ 
and we replace $\theta$ in  \eqref{riceq2.3}  by $ \frac 1 \gamma$.
 We have $$F\big(\frac 1\gamma\big)= f^{ \a }  x  f^{ \a  {(\gamma-1)} } .$$
  Note that $\frac 1R = \frac 1 p + \a (\gamma-1)$ so that 
by H\"older, multiplication by $f^{  \a (\gamma-1) }$ (left or right)  is of norm 1
from $L_p(\tau)$ to $L_R(\tau)$.
As $x=x^*$, the right hand side of  \eqref{riceq2.3} is exactly 
$\|xf^{\a \gamma}\|_R^2=\|f^{\a \gamma} x\|_R^2$. The $R$-triangle inequality gives (note that $R<p\leq 1$):
$$\|x f^{\a \gamma} \|_R^R=\|x f^{\a\gamma} W\|_R^R \leq 
\big\| F\big(\frac 1\gamma\big)\big\|_R^R  +\big\|(x W f^\a +Vf^\a x)f^{\a(\gamma-1)} \big\|_R^R\leq \big\| F\big(\frac 1\gamma\big)\big\|_R^R  +\big\|x W f^\a +Vf^\a x \big\|_p^R. $$
Let us assume for the moment that $\big\|x W f^\a +Vf^\a x \|_p\leq \| xf^{\a \gamma}\|_R$, so that by convexity of $t\mapsto t^{\frac 2R}$:
$$\big\| F\big(\frac 1\gamma\big)\big\|_R^2\geq  \big\|x f^{\a\gamma} \big\|_R^{2}- \frac 2 R \big\|x f^{\a\gamma} \big\|_R^{2-R} \big\|x W f^\a +Vf^\a x \big\|_p^R.$$
Thus, we get
 $$\big\|F\big\|^2_{L_R(\P_U^{\frac 1 \gamma};L_R(\tau))}-\big\|F\big(\frac 1 \gamma\big)\big\|^2_{L_R(\tau)}\leq \frac 2 R \big\|x f^{\a\gamma} \big\|_R^{2-R} \big\|x W f^\a +Vf^\a x \big\|_p^R.$$ 
  Let $u_t= f^{\gamma \a it}$.
  By   \eqref{riceq2.3}, we have (a fortiori)
  $$ \delta_R\left(\frac 1 {\gamma}\int \big\| f^{ \a }  x  f^{ \a (\gamma-1) }- u_t  f ^{\gamma \a} x u_t^*    \big\|^R_R Q_{\frac 1\gamma}^1(dt)  \right)^{\frac 2R}\leq 
 \frac 2 R \big\|x f^{\a\gamma} \big\|_R^{2-R} \big\|x W f^\a +Vf^\a x \big\|_p^R.$$
Put $H(t)=u_t  f ^{\gamma \a} x-f^{ \a }  x  f^{ \a (\gamma-1) }u_t$, so that under the assumption $\big\|x W f^\a +Vf^\a x \|_p\leq \| xf^{\a \gamma}\|_R$:
\begin{equation}\label{eqh}  \int \| H(t) \|^R_R Q_{\frac 1\gamma}^1(dt) \leq C_R 
  \big\|x W f^\a +Vf^\a x \big\|_p^{\frac {R^2}2} \big\|x\big\|_s^{R- \frac {R^2}2}.\end{equation}
We now introduce 
   the analytic function $G:\ U\to L_R(\tau)$ defined by
   $$G(z)=  Vf ^{ \gamma \a (1-z)} x+   x Wf ^{ \gamma\a (1-z)}.$$
We  apply \eqref{riceq2.9} at the point $\omega=\frac{\gamma +\theta-1}{\gamma} \in U$ with $p_1=s$ and $p_0=R$. Easy computations give
$$1-\omega= \frac {1-\theta}\gamma,\quad \frac 1{p_\omega} =\frac {1-\omega}R+\frac \omega s=\frac 1 q.$$
  Note that for any $t\in \R$ we have (recall  $\chi_p=\max\{2^{\frac{1}{p}-1},1\}$)
  $$\|G(1+it)\|_s\le 2\chi_s\|x\|_s.$$
For $t\in \R$, we also have
\begin{equation}\label{eq13} G(it)=Vu_{-t}f^{\gamma \alpha }x+xf^{\gamma \alpha}Wu_{-t}=VH(-t) +
(V f^\a x + xf^\a W)f^{\a (\gamma -1)}u_{-t}.\end{equation}
We want to estimate $\int \|G(it)\|_R^R Q^0_\omega(dt)$ in order to be able to apply
\eqref{riceq2.9}.\\
We will distinguish between two cases: either $\big\|x W f^\a
+Vf^\a x \big\|_p$ is larger than $\| xf^{\a \gamma}\|_R$ or not.\\
 If $\big\|x W f^\a
+Vf^\a x \big\|_p> \| xf^{\a \gamma}\|_R=\| f^{\a \gamma}x\|_R$, then
clearly (by the $R$-triangle inequality)
\begin{equation}\label{eq14} \int \big\|G(it)\big\|_R^R Q^0_\omega(dt) \leq 2\big\|x W f^\a +Vf^\a x\big\|_p^R. \end{equation}
Otherwise with the H\"older inequality, the $R$-triangle inequality and \eqref{eq13}:
    $$\int \big\|G(it)\big\|_R^R Q^0_\omega(dt)\le \int \big\|H(-t)\big\|_R^R
Q^0_\omega(dt) + \big\|Vf^\a x + x f^\a W \big\|_p^R.$$ 
But, by the explicit
formula for the Poisson kernels   on the strip (see \cite{BL}
page 93), namely
$$Q^k_{\omega}(t)= \frac {\sin{\pi \omega}}{2\big(\cosh(\pi t)- (-1)^k\cos(\pi\omega)\big)}, \quad (k=0,1, \ 0<\omega<1)$$
we know that $Q^1_{\frac 1 \gamma}(t)dt$ and  $Q^0_{\omega}(t)dt$ are equivalent symmetric measures. \\ Thus, in the second case, by \eqref{eqh}  we can find some constant $M_\gamma$ with 
$$\int \big\|G(it)\big\|_R^R Q^0_\omega(dt)\le M_\gamma \big\|x W f^\a +Vf^\a x \big\|_p^{\frac {R^2}2}\big\|x\big\|_s^{R- \frac {R^2}2}.$$
Since $ \|x W f^\a +Vf^\a x \big\|_p  \le \big\|x\big\|_s$, by \eqref{eq14} a similar bound also holds in the first case (and hence in both cases).
By \eqref{riceq2.9}  there is a constant $C_\gamma$ depending on 
$\gamma$ (and $p,\,q$ and $s$) so that:
\begin{eqnarray*}
\big\| Vf^{\alpha(1-\theta)}x+ xWf^{\alpha(1-\theta)}\big\|_q & \leq& C_\gamma 
\big\|x W f^\a +Vf^\a x \big\|_p^{\frac {R(1-\omega)}2} \big\|x\big\|_s^{(1-\frac R2)(1-\omega)+\omega}\\ & = &C_\gamma \big\|x W f^\a +Vf^\a x \big\|_p^{\frac {R(1-\theta)}{2\gamma}} \big\|x\big\|_s^{
1-\frac {R(1-\theta)}{2\gamma}}.
\end{eqnarray*}
The theorem is obtained by letting $\gamma\to 1$ (note that $C_\gamma\to \infty$).
  \end{proof}
  
\begin{rem}
 In Theorem \ref{nho} with the same notation, one can also get for $x\in L_s(\tau)$ with  $f,g\in L_1(\tau)^+$ of norm 1 and $V$  a unitary commuting with $f$ and $W$ a unitary commuting with $g$
\[\big\|x g^{\a(1-\theta)}W + Vf^{\a(1-\theta)}  x\big\|_q\le C\big\| x g^{\a }W+ Vf^{\a }  x   \big\|_p^{\frac{R}{2}(1-\theta)} \big\|x\big\|_s^{1-\frac{R}{2}(1-\theta)}.\]  
To see it, one can use the $2\times 2$ trick of \eqref{trick} but with  
$\tilde f =\frac 1 2
\left[\begin{array}{cc} f & 0 \\ 0 & g\end{array}\right]$.
\end{rem}  
    
\begin{rem}
 One can now deduce an estimate weaker than \eqref{sup5ter}, with the notation of
 Theorem \ref{nho}:
$$\max\big\{ \big\| f^\a x\big\|_p,\big\|xf^\a\big\|_p\big\} \lesssim \big\|f^\a x+xf^\a\big\|_p^{\frac R4}\big\|x\big\|_s^{1-\frac R 4}.$$  
This follows from the identity $2f^\a x= f^ax+xf^\a+f^{\frac \a 2}\big(f^{\frac \a 2}x+xf^{\frac \a 2}\big) -\big(f^{\frac \a 2}x+xf^{\frac \a 2}\big)f^{\frac \a 2}$ and Theorem \ref{nho} with $\theta =\frac 12$.
\end{rem} 
\section{ Maurey's factorization of operators with values in $L_p$, $0<p<1$ }\label{sec4}
In this section we attempt to give a non-commutative version of the following result
of Maurey \cite{Ma2}: Let $(\Omega,\mu)$ be any measure space and let $0<p< 1$.
For any bounded linear operator $u: H\to L_p(\mu)$ with $\|u\|\le 1$
there is a probability density $f$ such that (with the convention $\frac{0}{0}=0$)
$$\forall x\in H\quad \left(\int \Big|\frac {u(x)}{f^\alpha}\Big|^2 d\mu \right)^{\frac 12}\le C \|x\|,$$
where $C$ is a constant depending only on $p$ and where, as before, $\alpha=\frac1p-\frac 12$.
We refer to \cite{LP2,LPX,JuP} for results related to this section.

We will use the following well known variant of the Hahn-Banach theorem
(see e.g. \cite[p.39 and p.421]{P2}).
             
       \begin{lem}\label{hb}    Let $S$ be a set and
let ${\cal F}\subset
\ell_\infty(S)$ be a convex cone of real valued functions on $S$ such that
$$\forall\ f \in {\cal F}\qquad \sup_{s\in S} f(s) \ge 0.$$
Then there is a net $(\lambda_i)$ of finitely supported probability
measures on $S$ such that
$$\forall\ f\in {\cal F}\qquad \lim_{\cl U} \int fd\lambda_i \ge 0$$
for any ultrafilter $ {\cl U} $ refining the net.
\end{lem}

It will be convenient to use the following notation:

Let $f$ be any element in $(M_*)_+$, and as before $\alpha=\frac1p-\frac12$.
We denote by $J(f)$ ($J$ for Jordan) the mapping
 $x\mapsto \frac{fx+xf} 2$.
We will denote by $L(f)$ and $R(f)$ the left and right multiplications by $f$, 
 so that  $J(f)=\frac{L(f)+R(f)}2$. 
 We use the same notation for $$J(f^\alpha)(x)=\frac{f^{\alpha} x + x  f^{\alpha}}2 $$
 for any $\alpha>0$. Note that
 since $f^\alpha\in L_{\frac 1\alpha}(\tau)$ the latter mapping is bounded
 from $L_2(\tau)$ to $L_p(\tau)$ when $\alpha=\frac 1 p-\frac 12$.
 Moreover, this mapping preserves self-adjointness.\\
In addition, if $f$ has full support (\ie if $f$ is a faithful
 normal state on $M$) then the mapping $J(f^\alpha): L_2(\tau)\to L_p(\tau)$
is injective. We essentially already observed
this (see   Remark \ref{ultra}). Indeed, if $J(f^\alpha)(x)=0$ for $x\in  L_2(\tau)$
then the same is true for the real and imaginary parts  of $x$,
so we may assume $x=x^*$. Then $J(f^\alpha)(x)=0$ implies
$f^{2\alpha}x=x f^{2\alpha}$ and hence,
since $x$ is self-adjoint, $f^{\alpha}x=x f^{\alpha}$. But then
$J(f^\alpha)(x)=f^{\alpha}x$, and since $f$ has full support,
$f^{\alpha}x=0$ implies $x=0$.

In this section, we will assume 
 that
$M$ is $\sigma$-finite. Then (see \cite[p. 78]{Tak1}) there is a faithful
normal state $f_0$ on $M$.   It will be convenient to invoke the following elementary Lemma.

\begin{lem}\label{div}
Let $f\in {\cl D}$ and let $g= (f^{2\alpha}+f_0^{2\alpha})^{\frac 1{2\alpha}}$.
Then $g$ is faithful and $\tau(g)\le 2$. Moreover
$$J(f^\alpha) (L_2(\tau)) \subset J(g^\alpha) (L_2(\tau))$$
and for any $y\in L_2(\tau)$ we have
 \begin{equation}\label{mau1}\big\|J(g^\alpha)^{-1} J(f^\alpha) y\big\|_2\le 2\big\|y\big\|_2.\end{equation}
\end{lem}
\begin{proof} Since $g\ge f_0 $, $g$ is faithful. Let $s=\frac 1{2\alpha}$.
Note $s\le 1$. By the $s$-triangle inequality (see \eqref{tri2}), we have
 $\tau(g)\le \tau(f)+\tau(f_0)\le 2$ .
 We will prove more generally that
$$L( f^\alpha) (L_2(\tau) )+ R( f^\alpha) (L_2(\tau))\subset J(g^\alpha)( L_2(\tau)).$$
Note that  $f^{2\alpha} \le g^{2\alpha}$. Therefore, as unbounded operators on $L_2(\tau)$ we have
$$L(f^{2\alpha}) =(L(f^{\alpha}))^2 \le  (L(f^{\alpha}))^2 + 2 L(f^{\alpha})R(f^{\alpha}) +(R(f^{\alpha}))^2=  4 (J(g^{\alpha}))^2 .$$
This implies $$\|(J(g^{\alpha}))^{-1} L(f^{\alpha})  \|_{  L_2(\tau)\to L_2(\tau)} \le   2.$$
Thus if $x=f^\alpha y$ with $y\in L_2(\tau)$, we have $x=J(g^{\alpha})(y')$ with
$y'=(J(g^{\alpha}))^{-1} L(f^{\alpha})(y)\in  L_2(\tau)$ and $\|y'\|_2\le 2 \|y\|_2$.
A similar result holds for the right hand side multiplication.
This proves the announced inclusion and \eqref{mau1} follows by the triangle inequality
in $L_2(\tau)$.
\end{proof}

We will denote by ${\cl D}'$ the subset of ${\cl D}$ formed of
the elements with full support. Assuming $f\in {\cl D}'$  and $ T\in L_p(\tau)$,
we denote\\
\centerline{$D_{f^\alpha}(T)=\|(J(f^\alpha))^{-1} (T)\|_2$  \text{
if } $T\in J(f^\alpha)(L_2(\tau))$ }
\text{ and we set }
  {$D_{f^\alpha}(T)=\infty$}   \text{ if } $T$ \text{ is not in this range.}

\begin{thm}\label{gric} Let $0<p<1$.  There is a constant $C$ for which the following holds:
Let $u:\ H\to L_p(\tau)$ be a bounded operator with $\|u\|\le 1$.
Then there is a net of
 finitely supported  probability measures $(\lambda_i)$ on $\cl D'$
 such that for any $x\in H$ we have   
 \begin{equation}\label{ric11} \left(\lim_{\cl U} \int [D_{{f}^\alpha}(u(x))]^2 d\lambda_i(f)\right)^{\frac 12}\le C \|x\|_H,\end{equation}
where $\cl U$ is any   ultrafilter refining the net.
\end{thm}
\begin{proof}[Proof of Theorem \ref{gric}] 
For any finite sequence $(x_k)$ in $H$
we have by \eqref{eq0.1}
$$ |||(u(x_k)) |||_p 
\le \beta_p \left(\int \big\|\sum r_k(t) u(x_k)\big\|^p_pdt \right)^{\frac 1 p}$$
 $$\le \beta_p \left(\int \big\|\sum r_k(t)  x_k\big\|^2_H dt\right)^{\frac 12}=\beta_p \Big( \sum\big\| x_k\big\|_H^2 \Big)^{\frac 12}.$$
By Step 2 from  \S \ref{sec1} and by  
Lemma \ref{div} it follows that for some constant $\beta'_p$
we have $$\inf_{f\in {\cl D}'  } \sum D_{f^\alpha}(u(x_k))^2 \le\beta^{'2}_p \sum\big\| x_k\big\|_H^2 .$$
Let $ a_k>0$ be arbitrary coefficients.
We may obviously replace $x_k$ by $a_k x_k$.
 Let us now fix a  sequence $(T_k)$ in $L_p(\tau)$. Assume
 given $\beta>0$ such that  for any sequence $ a_k>0$ we have
 \begin{equation}\label{sup4} \inf_{f\in {\cl D}'  } \sum  a^2_k D_{f^\alpha}(T_k)^2 \le\beta  
 \sum a^2_k .\end{equation}
 Note that this obviously remains true if we replace the infimum over $ {\cl D}'  $
 by an infimum over  the set $ {\cl D}_T  $ formed of those
 $f\in   {\cl D}' $ such that  $D_{f^\alpha}(T_k)<\infty$ for all $k=1,\cdots,n$.
 By Lemma \ref{hb} there is a net of
 finitely supported  probability measures $(\lambda_i)$ on $\cl D_T$
 and an ultrafilter ${\cl U}$
 such that for any  $k=1,\cdots,n$ we have
$$\lim_{\cl U} \int [D_{{f}^\alpha}(T_k)]^2 \lambda_i(df)\le \beta.$$
Therefore, since \eqref{sup4} holds for $T_k=u(x_k)$ and $\beta=   \beta^{'2}_p$, for any finite sequence $x_k\in B_H$ 
there is a net of
 finitely supported  probability measures $(\lambda_i)$ on $\cl D'$
 such that for any  $k=1,\cdots,n$
  $$\lim_{\cl U} \int [D_{{f}^\alpha}(u(x_k))]^2 \lambda_i(df)\le \beta^{'2}_p    .$$
  Then after a simple rearrangement of this net 
  (e.g. by indexing our net by the set of finite subsets of $B_H$) we   obtain
  a net $(\lambda_i)$ of
 finitely supported  probability measures on ${\cl D}'$ such that
 $$\forall x\in H\quad \lim_{\cl U} \int [D_{{f}^\alpha}(u(x))]^2 \lambda_i(df)\le \beta^{'2}_p   \big\| x\big\|_H^2 .$$

\end{proof}
\begin{rem}\label{gjric}    When $f\ge 0$ is bounded,  
 the   multiplications $L(f)$ and $R(f)$ are self-adjoint nonnegative operators
 on $L_2(\tau)$ and hence the same is true for $J(f)$.
 Moreover for any $0\le g\le f$ we have $J(g)\le J(f)$
 and hence $J(f)^{-1}\le J(g)^{-1}$.
 Now if $p\ge 1$ since $2\alpha \le 1$ it follows
 that $x\mapsto x^{-2\alpha}$ is operator convex on $\R_+$ (see e.g. \cite{Bh}),
 so we can deduce from \eqref{ric11}
 that
 there is a net of  densities $f_i\in {\cl D}'$
 (namely $f_i=\int f d\lambda_i(f)$)   such that  for some constant C
for any $x\in H$
$$\lim_{\cl U}   [D_{{f_i}^\alpha}(u(x))]^2 \le  C  \big\| x\big\|_H^2 .$$
Then at least in the finite dimensional case, 
the net converges and we can  find a density $g\in {\cl D}'$ such that
we have, as in Maurey's original theorem,
$$    [D_{{g}^\alpha}(u(x))]^2 \le  C  \big\| x\big\|_H^2 .$$
In the commutative case, the map $x\mapsto x^{-2\alpha}$
being convex on $\R_+$ for any $\alpha$,
this argument works also for $p<1$. However,
since $x\mapsto x^{-2\alpha}$
is not operator convex on $\R_+$ when  $\alpha>\frac 12$,
we do not see how to  complete that same argument for $p<1$.
\end{rem}

\section{Application to the Mazur maps}\label{sec5}

 For a semifinite von Neumann algebra $(M,\tau)$, the Mazur map
 $M_{p,q}:L_p(\tau)\to L_q(\tau)$ is given by $M_{p,q}(f)=f|f|^{\frac
   {p-q}q}$. 
It is known to be a uniform homeomorphism on spheres (see \cite{Ray}) for 
$0< p,q<\infty$.
We would like to know for which
exponents $0<\gamma \le 1$ this (non-linear) map is $\gamma$-H\"older, by which we mean that
there is a constant $C$ such that
for any $g,h$ in the unit sphere of $L_p(\tau)$ we have
$$\big\| M_{p,q}(g)-M_{p,q}(h)\big\|_q\le C \big\| g-h\big\|_p^\gamma.$$
Precise estimates are given in \cite{R} in the 
case $1\leq p,q<\infty$, $M_{p,q}$ is H\"older with exponent 
$\min\{1,\frac p q\}$ as for commutative integration. Actually, it is shown there
that for $0<p,q<\infty$, $M_{p,q}$ is H\"older on all
semifinite von Neumann algebras  with exponent $\gamma$ and constant $C$ (both independent of $M$) iff the following inequalities occur for any finite von Neumann algebra $M$
\begin{equation}\label{acc} \forall x\in M,\, x=x^*,\, \|x\|_\infty=1,\; \forall f\in L_p(\tau)^+, \,\|f\|_p=1,
\quad \big\| xf^{\frac p q} \pm  f^{\frac p q}x \big\|_q\leq C'  \big\| xf \pm  fx \big\|_p^\gamma.\end{equation}

Thus we can use Theorem \ref{nho} to get

\begin{thm}
For any $0<p,q<\infty$ and any semifinite von Neumann algebra $(M,\tau)$, the Mazur map $M_{p,q}$ is $\gamma$-H\"older for
$\gamma<\frac 1{2q}\big(\frac {p}{3^k}\big)^2$ where $k\ge 0$ is the smallest integer
such that $\frac pq<3^k$.
\end{thm}

\begin{proof}
Using composition and \cite{R}, it suffices to do it when $0<p,q\leq 1$.

We start by looking at $0<p<q\leq 1$. By Theorem \ref{nho} with $s=\infty$, we have for all $R<p$ and some 
constant $C_R$, for $x$ and $f$ as in \eqref{acc}
\begin{equation}\label{eq4.0}\big\| xf^{\frac p q} \pm  f^{\frac p q}x \big\|_q\leq  C_R 
 \big\| xf \pm  fx \big\|_p^{\frac {pR}{2q}}.\end{equation}
Thus \eqref{acc} holds for any $\gamma<\frac {p^2}{2q}$, which gives us the case $k=0$.

To treat the case $q<p\leq 1$, note that for any $g,h\in L_p$ with norm 1, 
using $$\big\|g^*g-f^*f\big\|_{\frac p2}= \big\|g^*(g-f)+(g^*-f^*)f\big\|_{\frac p2}\le 2^{{\frac 2p}-1}\Big( \big\|g^*(g-f)\big\|_{\frac p2}+\big\|(g^*-f^*)f\big\|_{\frac p2}\Big)$$
and H\"older's inequality, we find
$$\big\| |g|^2- |h|^2 \big\|_{\frac p2} \leq 2^{\frac 2 p} \big\| g-h\big\|_p.$$ One easily deduces that $M_{p,\frac p 3}$ is $  1
$-H\"older. 
By iteration, the same is true for
$M_{p,\frac p {3^2}}=M_{\frac p 3,\frac p {3^2}}M_{p,\frac p 3} $
and more generally for $M_{p,\frac p {3^k}}$ for all $k\ge 0$.
Let $k$ so that $  \frac p{3^k}< q$. As
$M_{p,q}=M_{\frac p {3^k}, q}M_{p,\frac p {3^k}}$, one deduces that
    $M_{p,q}$  
    admits the same H\"older exponent as  $M_{\frac p {3^k}, q}$,
    which by \eqref{eq4.0} is  at least as good as $\gamma_k<\frac 1{2q}\big(\frac {p}{3^k}\big)^2$.
    Thus one concludes that $M_{p,q}$ is $\gamma_k$-H\"older.
    \end{proof}

\begin{rem} It seems likely that the exponents are   not optimal. For instance, one can argue as in Corollary 2.4
in \cite{R}, using Kosaki's inequality $\big\|g^p-h^p\big\|_1\leq \big\|g-h\big\|_p^p$ for $0<p<1$ and $g,h\in L_p^+$ to get that 
$$\big\| xf^p-f^px\big\|_1\leq 2\big\| xf-fx\big\|_p^p.$$
\end{rem}

\begin{rem} Because of the lack of convexity and in particular of
conditional expectations, one apparently cannot apply the Haagerup reduction technique to get the result for general von Neumann algebras.
\end{rem}

  \end{document}